\documentclass[11pt]{amsart}

\usepackage{amsmath,amssymb,amscd,amsthm,indentfirst}
\usepackage{amsfonts}

\usepackage[all]{xy}

\usepackage{graphicx}

\usepackage{ifpdf,ifdraft}
\ifdraft{

\newcommand{\assafsays}[1]{\marginpar{\textbf{A:} #1}}
}
{
\newcommand{\assafsays}[1]{}
}

\newtheorem{prop}{Proposition}[section]
\newtheorem{thm}[prop]{Theorem}
\newtheorem{cor}[prop]{Corollary}

\newtheorem{lem}[prop]{Lemma}

\theoremstyle{definition}
\newtheorem{defn}[prop]{Definition}

\newtheorem{remark}[prop]{Remark}
\newtheorem*{remarkx}{Remark}

\def\al{\alpha}
\def\be{\beta}
\def\de{\delta}
\def\ga{\gamma}
\def\De{\Delta}
\def\Ga{\Gamma}
\def\la{\lambda}
\def\si{\sigma}
\def\Si{\Sigma}
\def\th{\theta}
\def\vp{\varphi}

\def\C{\mathcal{C}}

\def\F{\mathcal{F}}

\def\I{\ensuremath{\mathcal{I}}}
\def\M{\ensuremath{\mathcal{M}}}
\def\O{\ensuremath{\mathcal{O}}}
\def\X{\ensuremath{\mathcal{X}}}

\def\RR{\ensuremath{\mathbb{R}}}
\def\ZZ{\ensuremath{\mathbb{Z}}}

\def\Ab{\mathbf{Ab}}

\def\spaces{\textbf{Spaces}}

\def\aut{\operatorname{aut}}
\def\Aut{\operatorname{Aut}}
\def\coker{\operatorname{coker}}
\def\colim{\operatorname{colim}}
\def\ev{\operatorname{ev}}
\def\Ext{\operatorname{Ext}}
\newcommand{\hholim}[1]{\underset{#1}{\operatorname{holim}}}
\def\hocolim{\operatorname{hocolim}}
\def\Hom{\mathrm{Hom}}
\def\id{\operatorname{id}}
\def\incl{\operatorname{incl}}

\newcommand{\llim}[1]{\underset{#1}{\operatorname{lim}}}
\def\map{\operatorname{map}}
\newcommand\mdl[1]{#1\operatorname{-mod}}
\def\op{\ensuremath{\mathrm{op}}}

\def\sd{\operatorname{sd}}
\def\Tor{\operatorname{Tor}}

\def\rk{\ensuremath{rk}}

\def\St{\operatorname{St}}

\newcommand{\ccolim}[1]{\underset{#1}{\operatorname{colim}}}
\newcommand{\hhocolim}[1]{\underset{#1}{\hocolim}}
\newcommand{\ul}[1]{\underline{#1}}
\newcommand{\xto}[1]{\xrightarrow{#1}}

\title[Self equivalences of linear spheres]{On the homotopy groups of the self equivalences of linear spheres}

\author{Assaf Libman}
\address{Institute of Mathematics, University of Aberdeen, Fraser Noble Building, Aberdeen AB24 3UE, U.K.}
\email{a.libman@abdn.ac.uk}

\renewcommand{\theenumi}{(\arabic{enumi})}
\renewcommand{\labelenumi}{(\arabic{enumi})}

\begin{document}

\begin{abstract}
Let $S(V)$ be a complex linear sphere of a finite group $G$. 
Let $S(V)^{*n}$ denote the $n$-fold join of $S(V)$ with itself and let $\aut_G(S(V)^*)$ denote the space of $G$-equivariant self homotopy equivalences of $S(V)^{*n}$.
We show that for any $k \geq 1$ there exists $M>0$ which depends only on $V$ such that $|\pi_k \aut_G(S(V)^{*n})| \leq M$ is for all $n \gg 0$.
\end{abstract}

\maketitle

\section{Introduction}

Let $\aut(X)$ denote the space of self homotopy equivalences of a space $X$ with the identity as a basepoint.
If $X$ is a $G$-space we write $\aut_G(X)$ for the equivariant self equivalences.
It is well known that for any $k \geq 1$ the sequence of groups 
\[
\pi_k \aut(S^0)  
\to \pi_k \aut(S^1) \to \dots \to 
\pi_k\aut(S^n) \xto{\pi_k(\vp \mapsto \Si \vp)} \pi_k \aut(S^{n+1}) \to \dots
\]
stabilizes on a finite group, namely for all $n \gg 0$ all the arrows become isomorphisms of finite groups.
The stable group is the stable $k$-homotopy group of $S^0$.
This can be deduced from the fibration $\Omega^nS^n \to \map(S^n,S^n) \to S^n$ together with \cite[19.1.2]{sln1662}) and the classical fact that the groups $\pi_{n+k}S^n$ stabilize on $\pi_k^S$. 

Equivariantly, one replaces $S^0$ with a sphere $X$ on which a finite group $G$ acts.
Let $X^{*n}$ denote the $n$-fold join of $X$ with itself equipped with the natural action of $G$.
We obtain a sequence of spaces
\[
\aut_G(X) \to \dots \to 
\aut_G(X^{*n}) \xto{\vp \mapsto \vp*1_X} \aut_G(X^{*(n+1)}) \to \dots
\]
If the $G$-sphere $X$ has a fixed point, the stabilization of $\{\pi_k\aut_G(X^{*n})\}_n$ where $k \geq 1$ can be deduced from the results of Hauschild \cite[Satz 2.4]{Hauschild}.
In the absence of fixed points, the problem is much harder.

A complex representation $V$ of a finite group $G$ admits a $G$-invariant scalar product, unique up to equivalence, and the subspace $S(V)$ of $V$ consisting of the unit vectors is called a \emph{linear sphere}.
The stabilization of $\{\pi_k\aut_G(S(V)^{*n})\}_n$ was established if $G$ is cyclic by Schultz \cite[Proposition 6.5]{Schultz73II}, or if $G$ acts freely on $S(V)$ by Becker and Schultz \cite{BeSch}.
The finiteness of the groups $\pi_k \aut_G(S(V)^{*n})$ for all sufficiently large $n$ was proven by Klaus \cite[Proposition 2.5]{Klaus} and \"Unl\"u-Yalcin \cite[Theorem 3.1]{UY12}.

\begin{defn}\label{def ess bdd}
A sequence of (abelian) groups $E_1,E_2,\dots$ is called \emph{essentially bounded} if there exists some $M >0$ such that $|E_n|\leq M$ for all $n \gg 0$.
\end{defn}

The main result of this paper is the following theorem.
It has its origin and motivation in the problem of constructing free actions of finite groups on products of spheres, see \cite{UY12}, \cite{Klaus}.
We will prove it in Section \ref{sec main result}.

\begin{thm}\label{thm main them aut}
Let $V$ be a complex representation of a finite group $G$.
Then for any $k \geq 1$ the sequence of groups
$\{\pi_k \aut_G(S(V)^{*n})\}_{n\geq 1}$ is essentially bounded.
\end{thm}

\begin{cor}
For any $k \geq 1$ the group $\varinjlim_n \, \pi_k \aut_G(S(V)^{*n})$ is finite.
\end{cor}

\section{Polytopes and isotropy groups}

Let $G$ be a finite group.
Let $X$ be a $G$-space
For any $x \in X$ let $G_x$ denote the isotropy group of $x$.

\begin{defn}\label{def x>K}
For any $K \leq G$ let $X^K$ denote the subspace of $X$ fixed by $K$.
Set
\[
X^{>K} := \{ x \in X \ \colon \ G_x > K\} = \bigcup_{K < H} X^H
\]
\end{defn}

When $K$ is the trivial subgroup, $X^{>e}$ is the subspace of $X$ consisting of the non-free orbits of $G$.

\begin{remark}\label{rem x>K 1}
Clearly $X^{>K} \subseteq X^K$ and both are invariant under the action of $N_GK$.
Hence, they both admit an action of $W=N_GK/K$.
Clearly, if we regard $X^K$ as a $W$-space, then $X^{>K} \supseteq (X^{K})^{>e}$.
If $K < H\leq N_G(K)$ and $X^{>K} \subseteq A \subseteq X^K$ is $W$-invariant then $A^H=X^H$ because $X^H \subseteq X^{>K} \subseteq A \subseteq X^K$. 
\end{remark}

Recall from \cite{Spanier66} that an abstract simplicial complex $X$ is a collection of non-empty subsets, called simplices, of an underlying set $V$ of ``vertices'', which contains all the singletons in $V$ and if $\si \in X$ is a simplex then any non empty $\tau \subseteq \si$ is also a simplex.
A \emph{polytope} is the geometric realization of a simplicial complex $X$.
By abuse of terminology we will not distinguish between a simplicial complex $X$ and the associated polytope.
A sub-polytope of $X$ is the geometric realization of a sub-complex.

Throughout, whenever a finite group $G$ acts on a polytope $X$, it is understood that it acts simplicially. 
By possibly passing to the barycentric subdivision $\sd X$ we may always assume that for any $K \leq G$, $X^K$ is a sub-polytope of $X$  on which $N_GK$ acts simplicially.
In particular, also $X^{>K}$ is a sub-polytope.
See \cite[Sec. 1]{Illman78} for more details.

\begin{lem}\label{lem abdy}
Let $W$ be a finite group acting simplicially on a finite polytope $Y$.
Let $A$ be a $W$-invariant sub-polytope of $Y$ which contains $Y^{>e}$.
Then, by possibly passing to the second barycentric subdivision $\sd^2(Y)$, there is a $W$-invariant open subset $U \subseteq Y$ which contains $A$ and 
\begin{enumerate}
\item
\label{lem abdy:1}
$\overline{U}$ and $B:=Y\backslash U$ are $W$-invariant sub-polytopes of $Y$.

\item
\label{lem abdy:2}
The inclusions $A \subseteq U \subseteq \overline{U}$ and $Y\backslash\overline{U} \subseteq B \subseteq Y\backslash A$ are $W$-homotopy equivalences. 
\end{enumerate}
\end{lem}

\begin{proof}
By \cite[Lemma 72.2]{Mu84} and the remarks in the beginning of \cite[\S 70]{Mu84}, the open subset $U:=\St(A,\sd^2(Y))$, namely the star of $A$ in the second barycentric subdivision of $Y$, has the properties in points \ref{lem abdy:1} and \ref{lem abdy:2} non-equivariantly.
Since $A$ is $W$-invariant, it is clear that  $U$ must be $W$-invariant, and hence so are $\overline{U}$ and $B$.
Thus, point \ref{lem abdy:1} follows.

To prove point \ref{lem abdy:2} we must show that for any $H \leq W$ the inclusions $A^H \subseteq U^H \subseteq \overline{U}^H$ and $(Y\backslash\overline{U})^H \subseteq B^H \subseteq (Y\backslash A)^H$ are homotopy equivalences.
If $H \neq 1$ then $A^H=U^H=\overline{U}^H$ and $(Y\backslash A)^H=\emptyset$ by Remark \ref{rem x>K 1} (with $K=1$).
If $H=1$ then these inclusions are homotopy equivalences by the choice of $U$.
\end{proof}

\section{Bredon homology and cohomology}\label{sec bredon cohomology}

In this section we will recall the definitions and some of the basic properties of Bredon homology and cohomology groups.
Most of the results are contained in Bredon's book \cite{BredonSLN34} in the cohomological setting.

Let $G$ be a finite group.
Let $\O_G$ denote the category of transitive left $G$-sets.
It is equivalent to its full subcategory whose objects are the left cosets $G/H$.
A \emph{cohomological coefficient functor} is a functor $\M \colon \O_G^\op \to \Ab$.
The cohomological coefficient functors form an abelian category with natural transformations as morphisms which will be denoted $\mdl{\O_G^\op}$.
It has enough injectives and projectives.
Similarly, a \emph{homological coefficient functor} is a functor $\M \colon \O_G \to \Ab$.
The abelian category of homological coefficient functors will be denoted $\mdl{\O_G}$; It has enough injectives and projectives.

A $G$-module $M$ gives rise to the following coefficient functors described in \cite[I.4]{BredonSLN34}.
For any $G$-set $\Omega$ let $\ZZ[\Omega]$ denote the permutation $G$-module whose underlying set is the free abelian group with $\Omega$ as a basis.
If $M$ is a left (resp. right) $G$-module, there is a cohomological (resp. homological) coefficient functor
\[
M \colon \Omega \mapsto \Hom_{\ZZ G}(\ZZ[\Omega],M), \qquad M \colon \Omega \mapsto M \otimes_{\ZZ G} \ZZ[\Omega].
\]
Associated to a cohomological coefficient functor $\M$ there is a unique equivariant cohomology theory, called \emph{Bredon cohomology}, defined on the category of $G$-CW complexes with the property that $H^*_G(\Omega;\M)=\M(\Omega)$ for any $G$-set $\Omega$ viewed as a discrete $G$-space.
See \cite[I.6]{BredonSLN34} for details.
When the coefficient functor $\M$ is associated with a $G$-module $M$ these cohomology groups have a particularly nice description as the homology groups of the cochain complex
\[
C^*_G(X;M) \, := \, \Hom_{\ZZ G}(C_*(X;\ZZ),M)
\]
where $C_*(X;\ZZ)$ is the ordinary (cellular) chain complex of $X$.
See \cite[I.9, p. I-22]{BredonSLN34}.
If $X$ is a polytope on which $G$ acts simplicially, then $H^*_G(X;\M)$ can be calculated by applying $\Hom_{\ZZ G}(-,M)$ to the simplicial chain complex of $X$. See \cite[\S 4.3]{Spanier66}.

In a similar way, one defines the Bredon homology groups $H_*^G(X;\M)$ with respect to a homological coefficient functor.
If $\M$ is associated with a right $G$-module $M$ then these groups are the homology groups of the chain complex
\[
C_*^G(X;M) \, := \, M  \otimes_{\ZZ G} C_*(X;\ZZ).
\]
If $X$ is a polytope on which $G$ acts simplicially then $C_*(X;\ZZ)$ can be replaced with the simplicial chain complex of $X$.

For any $G$-space $X$ and $n \geq 0$ there is an associated coefficient functor $\ul{H}_n(X;\ZZ)$ in $\mdl{\O_G^\op}$, see \cite[I.9]{BredonSLN34},
\[
\ul{H}_n(X;\ZZ) \colon G/H \, \mapsto \, H_n(X^H;\ZZ).
\]
In other words this functor has the effect $\Omega \mapsto H_n(\map_G(\Omega,X);\ZZ)$.
Given any cohomological coefficient functor $\M$ there is a first quadrant cohomological spectral sequence, see \cite[Section I.10, (10.4)]{BredonSLN34},
\[
E_2^{p,q}  = \Ext_{\mdl{\O_G^\op}}^p(\ul{H}_q(X;\ZZ),\M) \ \Rightarrow \ H_G^{p+q}(X;\M).
\]
Similarly, for any homological coefficient functor $\M$ there is a first quadrant homological spectral sequence
\[
E^2_{p,q}  = \Tor^{\mdl{\O_G}}_p(\ul{H}_q(X;\ZZ),\M) \ \Rightarrow \ H^G_{p+q}(X;\M).
\]
Recall that a map $f \colon X \to Y$ of $G$ spaces is a weak $G$-homotopy equivalence if it induces weak homotopy equivalences on the fixed points $X^H \simeq Y^H$ for all $H \leq G$.
The spectral sequences above imply the homotopy invariance of Bredon (co)homology.

\section{An equivariant Lefschetz duality}

The main result of this section is Proposition \ref{prop:lefshetz duality} which is an equivariant form of Lefschetz duality for Bredon (co)homology with respect to coefficient functors associated with trivial modules.
Its hypotheses should be compared with Lemma \ref{lem abdy}.

For any set $X$ let $\ZZ[X]$ denote the free abelian group with $X$ as a basis.
The assignment $X \mapsto \ZZ[X]$ is clearly functorial.
If $X$ is a (left) $G$-set then $\ZZ[X]$ is naturally a (left) $\ZZ G$-module.

\begin{lem}\label{lem nat isoms finite g sets}
Let $M$ be a trivial $\ZZ G$-module.
Then in the category of \emph{finite} $G$-sets there are isomorphisms, natural in $\Omega$
\begin{enumerate}
\item
$\Psi \colon \Hom(\ZZ[\Omega],\ZZ) \otimes_{\ZZ G} M \xto{\cong} \Hom_{\ZZ G}(\ZZ[\Omega],M)$
\label{nat isoms g sets:1}

\item
$\Theta \colon \Hom_{\ZZ G}(\Hom(\ZZ[\Omega],\ZZ),M) \xto{\cong} M \otimes_{\ZZ G} \ZZ[\Omega]$.
\label{nat isoms g sets:2}
\end{enumerate}
\end{lem}

\begin{proof}
The $G$-module $\ZZ[\Omega]$ has a canonical basis $\{\omega\}_{\omega \in \Omega}$.
For any abelian group $R$, any $\omega \in \Omega$ and any $r \in R$, let $\chi_\omega^R(r) \in \Hom(\ZZ[\Omega],R)$ denote the homomorphism determined by the assignment $\omega' \mapsto 0$ if $\omega' \neq \omega$ and $\omega \mapsto r$.
Since $\Omega$ is finite, the (right) $G$-module $\Hom(\ZZ[\Omega],\ZZ)$ has a canonical $G$-invariant basis $\chi_\omega:=\chi_\omega^\ZZ(1)$ where $\omega$ runs through $\Omega$.
For any $\omega \in \Omega$ let $G\omega$ denote the orbit of $\omega$ in $\Omega$.
Given $\Omega$ we will also choose a set $\{\omega_i\}$ of representatives to the orbits of $G$ on $\Omega$; the index $i$ runs through $\Omega/G$.
The homomorphisms $\Psi$ and $\Theta$ are defined by
\begin{eqnarray*}
\Psi & \colon & \chi_\omega \otimes_{\ZZ G} m \ \mapsto \ \sum_{\omega' \in G\omega} \chi_{\omega'}^M(m) \\
\Theta &\colon& \vp \ \mapsto  \ \sum_{i \in \Omega/G} \vp(\chi_{\omega_i}) \otimes_{\ZZ G} \omega_i.
\end{eqnarray*}
To see that $\Psi$ is well defined one uses the fact that $G$ acts trivially on $M$ and $\chi_\omega^R(r) \circ g^{-1} = \chi_{g\omega}^R(r)$.
For the same reasons 
$\Theta$ is independent of the choice of the representatives $\omega_i$.
By inspection $\Psi$ and $\Theta$ are natural with respect to $G$-maps $f \colon \Omega \to \Gamma$.
The details are left to the reader.
\end{proof}

For any polytope $X$ we will write $C_*(X)$ for the simplicial chain complex of $X$, see \cite[\S 4.3]{Spanier66}.
In the presence of a simplicial action of the group $G$, this becomes a chain complex of permutation $G$-modules, namely $C_n(X)$ is the $\ZZ G$-module $\ZZ[X_n]$ where $X_n$ is the $G$-set of the $n$-simplices of $X$.
If $G$ acts freely on $X$ then $C_*(X)$ is a chain complex of free $\ZZ G$-modules.
The cochain complex $C^*(X)$ is by definition $\Hom(C_*(X),\ZZ)$.
If $M$ is an abelian group then by definition 
\[
C_*(X;M)=C_*(X) \otimes M \qquad \text{and} \qquad C^*(X;M)=\Hom(C_*(X),M).
\]

\begin{prop}[Lefschetz duality]\label{prop:lefshetz duality}
Suppose that $G$ is a finite group acting simplicially on a finite $n$-dimensional polytope $Y$ which is a compact connected and orientable homology $n$-manifold. 
Let $U$ be an open $G$-subspace which contains $Y^{>e}$.
Assume that $\overline{U}$ and $B:=Y\backslash U$ are sub-polytopes of $Y$ and that the inclusion $Y\backslash\overline{U} \subseteq Y\backslash U$ is a homotopy equivalence.
Also assume that $G$ acts trivially on $H_n(Y)=\ZZ$ and that $H_n(Y) \to H_n(Y,\overline{U})$ is an isomorphism.
Set $D=B \cap \overline{U}$.
Then for any abelian group $M$  and any $p \geq 0$ there are isomorphisms
\begin{eqnarray*}
& & H^p_G(B;M) \  \cong \ H^G_{n-p}(B,D;M) \overset{\text{excision}}{\cong} H^G_{n-p}(Y,\overline{U};M) \qquad
\text{and} 
\\
& & H^G_p(B;M) \cong H_G^{n-p}(B,D;M)  \overset{\text{excision}}{\cong}  H_G^{n-p}(Y,\overline{U};M).
\end{eqnarray*}
\end{prop}

\begin{proof}
Throughout we let $C_*(Y), C_*(B), C_*(B,D)$ etc. denote the simplicial chain complexes of these finite polytopes.
For every $g \in G$ we obtain an automorphism $g_\#$ of $C_*(B)$ which in every degree $p$ permutes the basis elements of $C_p(B)$, namely permutes the set of $p$-simplices of $B$. 
Similarly there is an automorphism $g^\#$ of $C^*(B,D)$  which has the following effect on a $p$-cochain $c^p$
\[
g^\#(c^p) \colon \si \mapsto c^p(g_\#(\si)), \qquad \text{$\si$ is a $p$-simplex of $B$.}
\]
In this way $C_*(B)$ and $C^*(B,D)$ become (co)chain complexes of left $\ZZ G$-modules where
\[
g\cdot c_p = g_\#(c_p), \qquad \text{and} \qquad 
g \cdot c^p = (g^{-1})^\#(c^p).
\]
Since $U \supseteq Y^{>e}$, the action of $G$ on $B$ and $D$ is free and therefore $C_*(B)$ and $C^*(B,D)$ are (co)chain complexes of finitely generated free $\ZZ G$-modules.

Let $\Ga$ be an orientation cycle for $Y$, namely $\Ga$ is an $n$-cycle of $Y$ whose image $[\Ga] \in H_n(Y)=\ZZ$ is a generator.
Let $[\Ga_U] \in H_n(B,D)$ be the image of $[\Ga]$ under the isomorphism
\[
H_n(Y) \xto{\ \cong \ } H_n(Y,\overline{U}) \xrightarrow[\text{excision}]{\cong} H_n(B,D)
\]
For dimensional reasons, $Z_n(Y)=H_n(Y)$.
Since $G$ acts trivially on $H_n(Y)$, it follows that $\Ga$ is $G$-invariant.
Hence, $[\Ga_U]$ is $G$-invariant and therefore its preimage $\Ga_U \in Z_n(B,D)$ is also a $G$-invariant orientation cycle since $Z_n(B,D) = H_n(B,D)$.
By \cite[Theorem 66.1]{Mu84} the assignments 
\[
\Phi^p  \colon C^p(B,D) \xto{ \ \ c^p \, \mapsto \, (-1)^{n-p}\cdot ( c^p \cap \Ga_U) \ \ } C_{n-p}(B)
\]
form a morphism of cochain complexes (we view $C_{n-*}(B)$ as a cochain complex).
It follows from the $G$-invariance of $\Ga_U$ and from the naturality statement in \cite[Theorem 66.1]{Mu84} that $\Phi$ is a morphism of cochain complexes of $\ZZ G$-modules because for any $p$, set $\epsilon=(-1)^{n-p}$ and then
\[
g \cdot \Phi(g^{-1}\cdot c^p) = 
\epsilon \cdot
g_\#( g^\#(c^p) \cap \Ga_U) =
\epsilon \cdot
c^p \cap g_\#(\Ga_U) = 
\epsilon \cdot
c^p \cap \Ga_U = 
\Phi(c^p).
\]
We now apply \cite[Theorem 70.6]{Mu84} to the inclusion $D \subseteq B$ and the orientation class $\Ga_U$, and use excision together with the fact that $j_* \colon H_*(Y\backslash\overline{U}) \to H_*(Y\backslash U)$ is an isomorphism by hypothesis on $Y\backslash\overline{U} \subseteq Y\backslash U$, to deduce that $\Phi$ induces an isomorphism in homology.

Set $R_*(B,D):=C^{n-*}(B,D)$.
Thus, $R_*(B,D)$ is a chain complex which is obtained from the cochain complex $C^*(B,D)$ by simply re-indexing the modules.
So 
\[
\Phi \colon R_*(B,D) \to C_*(B)
\]
is a morphism of chain complexes of finitely generated free $\ZZ G$-modules.
It follows from K\"unneth's spectral sequence \cite[Theorem 5.6.4]{Weibel94} that the maps below induce isomorphism in (co)homology
\begin{eqnarray*}
& & 
M \otimes_{\ZZ G} R_*(B,D)  \xto{ \ M \otimes_{\ZZ G} \Phi   \ } M \otimes_{\ZZ G} C_*(B)  \\
\label{lefshetz:eq2}
& & 
\Hom_{\ZZ G}(C_*(B),M) \xto{ \ \Hom_{\ZZ G}(\Phi,M)  \ } \Hom_{\ZZ G}(R_*(B,D),M).
\end{eqnarray*}
By applying Lemma \ref{lem nat isoms finite g sets} to $C_*(B,D)$ and $C^*(B,D)$ and using the notation and results in Section \ref{sec bredon cohomology}, we obtain isomorphisms of chain complexes
\begin{eqnarray*}
& & M \otimes_{\ZZ G} R_*(B,D)  = M \otimes_{\ZZ G} C^{n-*}(B,D)  \cong C_G^{n-*}(B,D;M) \\
& & \Hom_{\ZZ G}(R_*(B,D),M) = \Hom_{\ZZ G}(C^{n-*}(B,D),M) \cong C_{n-*}^G(B,D;M)
\end{eqnarray*}
Therefore $H^{n-p}_G(B,D;M) \cong H^G_p(B;M)$ and $H_G^p(B;M) \cong H^G_{n-p}(B,D;M)$.
\end{proof}

\section{The key lemma}\label{sec key lemma}

Throughout this section we will fix a finite group $G$ and a sequence $\{ X_n\}_{n \geq1}$ of compact polytopes on which $G$ acts simplicially.
By possibly passing to the barycentric subdivisions, we may assume that for any $H \leq G$, the subspaces $(X_n)^H$ are sub-polytope of $X_n$.
We will make the following assumptions on $\{ X_n \}_n$.
Important examples are given by $X_n=S(V)^{*n}$ where $S(V)$ is a linear sphere; See Propositions \ref{prop lin spheres good} and \ref{prop abc 123}.

\renewcommand{\theenumi}{(\Roman{enumi})}
\renewcommand{\labelenumi}{(\Roman{enumi})}
\begin{enumerate}
\item
\label{key I}
For any $H \leq G$ either $X_n^H$ are empty for all $n \gg 0$, or for any $n \gg 0$ these are connected compact and orientable homology $N$-manifolds for some $N$ (which depends on $n$) such that $N_G(H)$ acts trivially on $H_N(X_n^H;\ZZ)\cong \ZZ$.

\item
\label{key II}
For any $H \leq G$,
if $i \geq 1$ then $H_i(X_n^H;\ZZ)=0$ for all $n \gg 0$.   

\item
\label{key III}
If $H' \leq H$ then either 
\begin{itemize}
\item[(i)] $\lim_{n \to \infty} (\dim X_n^{H'} - \dim X_n^H)= \infty$, or
\item[(ii)] $X_n^H=X_n^{H'}$ for all $n \gg 0$.
\end{itemize}
\end{enumerate}
\renewcommand{\theenumi}{(\arabic{enumi})}
\renewcommand{\labelenumi}{(\arabic{enumi})}

Let us now fix a subgroup $K \leq G$ and set $W=N_GK/K$.
For every $n\geq 1$ set $Y_n=X_n^K$ and $A_n=X_n^{>K}$ (Definition \ref{def x>K}).
These are polytopes on which $W$ acts simplicially.
By Remark \ref{rem x>K 1}, $A_n \supseteq (Y_n)^{>e}$.

By Proposition \ref{lem abdy}, for every $n \geq 1$ we can choose a $W$-invariant neighbourhood $U_n \subseteq Y_n$ of $A_n$ such that $\overline{U_n}$ and $B_n:=Y_n \backslash U_n$ are sub-polytopes of $Y_n$, and the inclusions $A_n \subseteq U_n \subseteq \overline{U_n}$ and $Y_n \backslash \overline{U_n} \subseteq B_n \subseteq Y_n \backslash A_n$ are $W$-homotopy equivalences.
Set $D_n=B_n \cap \overline{U_n}$.

\begin{lem}\label{lemma key}
Let $G$ be a finite group.
Let $\{X_n\}_n$ be a sequence of compact $G$-polytopes satisfying \ref{key I}--\ref{key III} above.
Fix $K \leq G$, set $W=N_GK/K$, and let $A_n, U_n, B_n, D_n \subseteq X_n$ be the subspaces defined above.
Assume further that $A_n \subsetneq Y_n$ for all $n \gg 0$.
Write $N$ for the dimension of $Y_n$.
Then
\renewcommand{\theenumi}{(\Alph{enumi})}
\renewcommand{\labelenumi}{(\Alph{enumi})}
\begin{enumerate}
\item
\label{key A}
For any abelian group $T$ and any $k \geq 0$ there are isomorphisms $H^W_k(B_n;T) \cong H_k(W;T)$ for all $n \gg 0$.
The right hand side is group homology with the trivial $W$\!--module $T$.

\item
\label{key B}
For any finite abelian group $R$ and any $k \geq 0$, the sequences of groups $\{ H_k^W(Y_n;R)\}_n$ and $\{ H_k^W(A_n;R)\}_n$ are essentially bounded.

\item
\label{key C}
If $k \geq 1$ then $\{ H_W^{N-k}(Y_n;\ZZ)\}_n$ and  $\{ H_W^{N-k}(Y_n,A_n;\ZZ)\}_n$ are essentially bounded.
If $R$ is a finite abelian group then $\{ H_W^{N-k}(Y_n;R)\}_n$ and  $\{ H_W^{N-k}(B_n;R)\}_n$ and $\{ H_W^{N-k}(D_n;R)\}_n$ are essentially bounded for any $k \geq 0$.

\item
\label{key D}
Consider the homomorphisms $H_W^{i}(D_n;\ZZ) \xto{\la_i} H_W^i(B_n;\ZZ)$ induced by $D_n \subseteq B_n$.
Then $\{\ker \la_{N-k}\}_n$ is essentially bounded for any $k \geq 1$ and $\{\coker \la_{N-k}\}_n$ is essentially bounded for any $k \geq 2$.
\end{enumerate}
\renewcommand{\theenumi}{(\arabic{enumi})}
\renewcommand{\labelenumi}{(\arabic{enumi})}
\end{lem}

\begin{proof}
For all $n \gg 0$, $Y_n\neq \emptyset$ since $A_n \subsetneq Y_n$.
Since $\dim(A_n)$ is the maximum of $\dim(X_n^H)$ where $H>K$, hypothesis \ref{key III} implies
\begin{equation}\label{eq lim codim}
\lim_{n \to \infty} (\dim Y_n - \dim A_n) = \infty.
\end{equation}

\noindent
Proof of \ref{key A}:
Recall that $N$ denotes $\dim Y_n$.
Given $0 \leq i \leq k$, we obtain the isomorphisms below for all $n \gg 0$.
The first isomorphism follows from the equivalence $Y_n\backslash U_n \simeq Y_n\backslash A_n$ and from Lefschetz duality \cite[Theorem 70.2]{Mu84} for $A_n \subseteq Y_n$ which is applicable by hypothesis \ref{key I} since $Y_n=X_n^K$.
The second isomorphism follows from \eqref{eq lim codim}, the third from Poincar\'e duality, and the fourth from hypotheses \ref{key II} and \ref{key I} since $Y_n \neq \emptyset$.
\begin{multline*}
H_i(B_n;\ZZ) \cong
H^{N-i}(Y_n,A_n;\ZZ) \cong 
H^{N-i}(Y_n;\ZZ) \\ \cong
H_i(Y_n;\ZZ) \cong 
\left\{
\begin{array}{ll}
\ZZ & \text{if } i=0 \\
0 & \text{if } 1 \leq i \leq k
\end{array}\right.
\end{multline*}
By construction, $W$ acts freely on $B_n$, so $C_*(B_n)$ is a chain complex of finitely generated free $W$-modules. 
We obtain a K\"unneth's spectral sequence \cite[Theorem 5.6.4]{Weibel94} for $X_*^W(B_n;T)=C_*(B_n) \otimes_{\ZZ W} T$
\[
E^2_{p,q}(n) = \Tor_p^{\ZZ W}(H_q(B_n);T) \Rightarrow H^W_{p+q}(B_n;T)
\]
We have seen above that if $n \gg 0$ then $H_q(B_n)=0$ for any $1 \leq q \leq k$ and therefore we obtain isomorphisms, for all $0 \leq i \leq k$
\[
H^W_i(B_n;T) \cong \Tor_i^{\ZZ W}(H_0(B_n),T) = \Tor_i^{\ZZ W}(\ZZ,T)=H_i(W;T).
\]

\noindent
Proof of \ref{key B}:
Consider the coefficient functors $\ul{H}_j(Y_n)$ defined in Section \ref{sec bredon cohomology}.
Given $i>0$, hypothesis \ref{key II} implies that $\ul{H}_i(Y_n)=0$ for all $n \gg 0$ and, that the sequence of functors $\{\ul{H}_0(Y_n)\}_n$ stabilizes on a coefficient functor $F \colon \O_W^\op \to \Ab$ whose values are the groups $0$ or $\ZZ$ (depending on whether $X_n^H$, where $K \leq H \leq N_G(H)$, are connected or empty for all $n \gg 0$).
As a consequence, for all $n \gg 0$, the spectral sequence
\[
E_{i,j}^2(n)=\Tor^{\mdl{\O_W}}_i (\ul{H}_j(Y_n),R) \ \Rightarrow \  H_{i+j}^W(Y_n;R)
\]
vanishes for $1 \leq j \leq k$ and $E_{i,0}^2(n)=\Tor^{\mdl{\O_W}}_i (F,R)$.
Hence $H_k^W(Y_n;R) \cong \Tor^{\mdl{\O_W}}_k (F,R)$ whose order is bounded by $|R|^\al$ for some $\al$ which depends only  on $\O_W$ by Lemma \ref{lem bdd cobar}.
Thus, $\{ H_k^W(Y_n;R) \}_n$ is essentially bounded. 

We now prove the second assertion of point \ref{key B}.
The first step is to show that there is some number $\be \geq 1$ such that for any $0 \leq i \leq k$ we have $\rk H_i(A_n)\leq \be$  for all $n \gg 0$.
Let $\I$ denote the poset of all the subgroup $H \leq G$ such that $H > K$.
For every $n$ there is a functor $F_n \colon \I^\op \to \spaces$ given by $F_n(H)=X_n^H$.
By hypothesis \ref{key II}, for all sufficiently large $n$, if $0 \leq j \leq k$ then $H_j(F_n)=0$, and the sequence of functors $H_0(F_n)$ stabilize of a functor $F' \colon \I^\op \to \Ab$ whose values are the groups $0$ or $\ZZ$.

Note that $F_n(H) \cap F_n(H')=F_n(\langle H,H'\rangle)$ and since $A_n=\cup_{H \in \I}F_n(H)$, it follows that $A_n=\ccolim{\I^\op} F_n$.
In fact since $X_n^H$ are polytopes, this also shows that the functors $F_n$ are Reedy cofibrant in the sense of \cite[Section 22]{DHKS}.
Therefore the natural maps
\[
\hhocolim{\I^\op} F_n \to \ccolim{\I^\op} F_n = A_n
\]
are homotopy equivalences.
We obtain a Bousfield-Kan spectral sequence
\[
E_{i,j}^2(n)=\colim^{\mdl{\I^\op}}_i H_j(F_n;\ZZ) \ \Rightarrow \ H_{i+j}(A_n;\ZZ).
\]
We have seen that if $n \gg 0$ then $E^2_{i,j}(n)=0$ for all $1 \leq j \leq k$ and therefore $H_i(A_n) \cong \colim^{\mdl{\I^\op}}_i F' = \Tor^{\mdl{\I^\op}}_i(\ZZ,F')$ for all $0 \leq i \leq k$. 
It follows from Lemma \ref{lem bdd cobar}(b) that there is $\be  \geq 1$ such that $\rk H_i(A_n) \leq \be$ for any $0 \leq i \leq k$ provided $n \gg 0$, 

For every $j \geq 0$ consider the coefficient functors $\ul{H}_j(A_n) \colon H \mapsto H_j(A_n^H)$ defined in section \ref{sec bredon cohomology} (here $H \leq W$).
If $H \neq 1$ then by Remark \ref{rem x>K 1}, $A_n^H=X_n^{\tilde{H}}$ where $\tilde{H}$ is the preimage of $H \leq W$ in $G$.
By hypothesis \ref{key II}, $H_j(A_n^H)=0$ for all $1 \leq j \leq k$ provided $n \gg 0$.
Also, the sequence $\{ H_0(A_n^H)\}_n$ stabilizes on a either $\ZZ$ or $0$.
If $H=1$ then we have seen that $\rk H_j(A_n^H) \leq \be$ for all $0 \leq j \leq k$ provided $n \gg 0$.
Thus, for any $n \gg 0$, if $0 \leq j \leq k$ then $\rk (\ul{H}_j(A_n)(-)) \leq \be$.
Applying Lemma \ref{lem bdd cobar}(a) to the spectral sequence (see section \ref{sec bredon cohomology})
\[
E^2_{i,j}(A_n) = \Tor^{\mdl{\O_W}}_i(\ul{H}_j(A_n);R) \ \Rightarrow \ H^W_{i+j}(A_n;R)
\]
we deduce that there is $\al >0$ such that for any $0 \leq i,j \leq k$ we have $|E^2_{i,j}(A_n)| \leq |R|^{\al \be}$ provided $n \gg 0$.
As a result $|H^W_k(A_n;R)| \leq |R|^{(k+1)\al\be}$ for all $n \gg0$, namely $\{H_k^W(A_n;R)\}_n$ is essentially bounded.

\noindent
Proof of \ref{key C} and \ref{key D}:
Let $M$ be an abelian group.
Since $Y_n \neq \emptyset$, it is, by hypothesis \ref{key I} a compact connected orientable homology $N$-manifold.
For any $0 \leq i \leq k$, equation \eqref{eq lim codim}, the choice of $U_n \supseteq A_n$, Proposition \ref{prop:lefshetz duality}, excision and point \ref{key A}, yield the following isomorphisms for all $n \gg 0$
\begin{multline*}
H_W^{N-i}(Y_n;M) \cong 
H_W^{N-i}(Y_n,A_n;M) \cong 
H_W^{N-i}(Y_n,\overline{U_n};M) \cong \\
H^{N-i}_W(B_n,D_n;M) \cong H_i^W(B_n;M) \cong H_i(W;M).
\end{multline*}
If $k \geq 1$ and $M=\ZZ$ then $|H_k(W;\ZZ)| \leq |W|^{|W|^k}$ by Remark \ref{rem bdd tor}, and therefore $\{H_W^{N-k}(Y_n;\ZZ)\}_n$ is essentially bounded.
This is the first part of point \ref{key C}.
Similarly, if $k \geq 0$ and $M$ is a finite group then $|H_k(W;M)| \leq |M|^{|W|^k}$, and therefore $\{H_W^{N-k}(Y_n;M)\}_n$ is essentially bounded.

The inclusions $A_n \subseteq \overline{U_n}$ are $W$-equivalences and, from \eqref{eq lim codim} it follows that $H_W^{N-k}(\overline{U_n};M)=0$ for all $n \gg 0$.
For every $n \geq 1$ we have $Y_n=B_n \cup_{D_n} \overline{U_n}$, so provided $n \gg 0$, the Mayer--Vietoris sequence yield the exact sequence
\begin{multline*}
\dots \to H_W^{N-k}(Y_n;M) \to H_W^{N-k}(B_n;M) \xto{\ \la_{N-k} \ } H_W^{N-k}(D_n;M) \to \\
\to H_W^{N-k+1}(Y_n;M) \to \dots
\end{multline*}
If $M=\ZZ$ then point \ref{key D} now follows from the first part of point \ref{key C} which we have proven above.
It remains to prove the second part of point \ref{key C}.
We have shown above that if $M$ is finite then $\{H_W^{N-i}(Y_n;M)$ is essentially bounded for any $i \geq 0$. 
Proposition \ref{prop:lefshetz duality} gives the isomorphisms
\[
H_W^{N-k}(B_n;M) \cong H^W_k(B_n,D_n;M) \cong H_k^W(Y_n,\overline{U_n};M) \cong H^W_k(Y_n,A_n;M)
\]
and $\{ H^W_k(Y_n,A_n;M) \}_n$ is essentially bounded by point \ref{key B} and the long exact sequence in homology.
Hence $\{ H_W^{N-k}(B_n;M) \}_n$ is essentially bounded.
The Mayer--Vietoris sequence above shows that also $\{ H_W^{N-k}(D_n;M) \}_n$ is essentially bounded.
\end{proof}

\begin{lem}\label{lem bdd cobar}
Let $\C$ be a finite category and $k \geq 0$ an integer.
Let $\al$ be the number of sequences of $k$ composable morphisms in $\C$.
Let $F \colon \C^\op \to \Ab$ be a functor such that $\rk (F(C)) \leq r$ for any $C \in \C$.
Let $G \colon \C \to \Ab$ be a functor.
\begin{itemize}
\item[(a)]
If there is some $M>0$ such that $|G(C)| \leq M$ for all $C \in \C$ then $|\Tor^{\mdl{\C}}(F,G)| \leq M^{r\al}$.

\item[(b)]
If $\rk (G(C)) \leq m$ for all $C \in \C$, then $ \rk \Tor^{\mdl{\C}}(F,G) \leq \al m r$.
\end{itemize}
\end{lem}

\begin{proof}
The groups $\Tor^{\mdl{\C}}(F,G)$ are the homology groups of a chain complex (the cobar construction) whose $n$th group has the form
\[
\bigoplus_{C_0 \to \dots \to C_n} G(C_0) \otimes F(C_n).
\]
Point (a) follows since $|G(C_0) \otimes F(C_k)| \leq M^r$ and (b) since $\rk (G(C_0) \otimes F(C_k) )\leq rm$.
\end{proof}

\begin{remark}\label{rem bdd tor}
As a consequence we see that if $G$ is a finite group and $M$ is an abelian group of rank $r$ then $\rk H_k(G;M) \leq r |G|^k $.
It follows that $|H_k(G;\ZZ)| \leq |G|^{(|G|^k)}$ for any $k \geq 1$ since $|G|$ annihilates $H_k(G;\ZZ)$.
Also, if $M$ is a finite group then $|H_k(G;M)| \leq |M|^{(|G|^k)}$.
\end{remark}

\section{The Barrat-Federer spectral sequence}
\label{sec federer ss}

The purpose of this section is to prove Theorem \ref{thm federer ss} below.
It is a special case of the Barrat-Federer spectral sequence tailored to our purposes.

Recall that a space $Y$ is called \emph{simple} if it is path connected and for any choice of basepoint, $\pi_1Y$ acts trivially on $\pi_*Y$.
In this case basepoints can be ignored in the sense that the basepoint-change isomorphisms for $\pi_qY$ are canonical.
Put differently, if $q \geq 1$ then $\pi_qY$ can be identified with the set of (unpointed) homotopy classes of (unpointed) maps $S^q \to Y$, which in this case has a natural group structure.
Hence, if a group $G$ acts on $Y$ then $\pi_qY$ has a natural structure of a $G$-module.

For $G$-spaces $X$ and $Y$ let $\map_G(X,Y)$ denote the space of $G$-maps.
Write $\map_G(X,Y)_f$ for the path component of $f \colon X \to Y$ with $f$ as a basepoint.

\begin{thm}\label{thm federer ss}
Let $X$ be a finite dimensional polytope on which a finite group $G$ acts freely and simplicially.
Let $Y$ be a simple $G$-space and fix a $G$-map $f \colon X \to Y$.
Then there exists a second quadrant homological spectral sequence
\[
E^2_{-p,q}=H^{p}_G(X;\pi_q Y) \Rightarrow \pi_{q-p} \map_G(X,Y)_f, \qquad (0 \leq p \leq q).
\]
The differential in the $E^r$-page has degree $(-r,r-1)$ and the $E^\infty_{-p,q}$-terms, where $q-p=k$, are the quotients of a finite filtration 
of $\pi_k \map_G(X,Y)_f$. 
\end{thm}

\begin{proof}
By possibly passing to the barycentric subdivision, we may assume that if $\Sigma \subseteq X$ is a simplex then $g\Si \cap \Si = \emptyset$ for any $1 \neq g \in G$.
Thus, if $\si$ is the orbit of a $k$-simplex in $X$ then $\si \cong G \times \De^k$.
Let $D$ be the subposet of the poset of all the $G$-invariant compact subspaces of $X$ whose objects are the orbits of the simplices of $X$.
We obtain a functor 
\[
\X \colon D \to G\operatorname{-}\spaces, \qquad \X \colon \si \mapsto \si.
\]
If $\si, \si' \in D$ and $\si \cap \si' \neq \emptyset$ then $\si \cap \si'$ is the orbit of a simplex which is the intersection of some simplices $\Si \subseteq \si$ and $\Si' \subseteq \si'$.
It follows that $X=\ccolim{D}\X$ and that $\X$ is Reedy cofibrant, \cite[Section 22]{DHKS}.
Therefore, the following natural map is a homotopy equivalence.
\begin{equation}\label{eq:colimx}
\hhocolim{D} \X \xto{\simeq}  X
\end{equation}
Since $G$ acts freely on both space, this is, in fact, a $G$-homotopy equivalence.
Therefore, for any $G$-module $M$ we obtain a Bousfield-Kan spectral sequence 
\[
E_2^{i,j} = \llim{D^{\op}}{}^i \, H_G^j(\X;M) \ \Rightarrow \  H^{i+j}_G(X;M).
\]
Every $\si \in D$ has the form $G \times \De^k$ so the natural transformation of functors $\X \to \pi_0\X$ has the property that $\X(\si) \to \pi_0\X(\si)$ are $G$-equivalences.
By the properties of $H_G^*(-;M)$ we obtain an isomorphism 
\[
H^j_G(\X;M) \cong H_G^j(\pi_0\X;M) \cong 
\left\{
\begin{array}{ll}
\Hom_{\ZZ G}(\ZZ[\pi_0\X], M) & \text{if } j= 0 \\
0 & \text{if } j>0
\end{array}
\right.
\]
where $\ZZ[\pi_0\X] \colon D \to \Ab$ is the functor $\si \mapsto \ZZ[\pi_0\X(\si)]$.
The Bousfield-Kan spectral sequence collapses and we obtain the isomorphism
\begin{equation}\label{eq:limphg}
H^p_G(X;M) = \lim{}_{\mdl{D^{\op}}}^p \, \Hom_{\ZZ G}(\ZZ[\pi_0\X],M), \qquad (p \geq 0)
\end{equation}
By applying the functor $\map_G(-,Y)$ to $\X$ we obtain a functor
\[
\map_G(\X,Y) \colon D^\op \to \spaces.
\]
The inclusions $\X(\si) \subseteq X$ give rise to $\map_G(X,Y) \to \map_G(\X(\si),Y)$ which carry $f$ to $f|_{\si}$.
Thus, the map $f$ gives rise to a consistent choice of basepoints in the functor $\map_G(\X,Y)$.
In other words, this functor can be viewed as a functor of \emph{pointed} spaces.
Also, \eqref{eq:colimx} implies the homotopy equivalence
\[
\map_G(X,Y) \simeq 
\map_G(\hhocolim{D}\, \X,Y) = \hholim{D^{\op}} \, \map_G(\X,Y)
\]
The category $D^\op$ is a poset of dimension $d=\dim X$, hence $\llim{D^{\op}}{}^p(-)$ vanishes if $p>d$.
Also $\pi_0Y=0$ and $\pi_1Y$ is abelian by assumption.
We can now apply \cite[XI.7.1 and IX.5.4 and X.7.1]{BK} to obtain a second quadrant homological spectral sequence
\[
E^2_{-p,q} = \lim{}_{\mdl{D^\op}}^{p} \, \pi_q \map_G(\X,Y) \, \Rightarrow \, \pi_{-p+q}\map_G(X,Y)_f, \quad (0 \leq p \leq q)
\]
with differentials of degree $(-r,r-1)$ in the $E^r$-page.
Since $E^2_{-p,q}=0$ for $p>d$ the $E^\infty_{-p,q}$-terms where $q-p=k$ are the filtration quotients of a filtration $\F_{-d,d+k} \subseteq \F_{-d+1,d+k-1} \subseteq \dots \subseteq \F_{0,k}$ of $\pi_k\map_G(X,Y)_f$.
It only remains to identify the $E^2$-page with the Bredon cohomology groups in the statement of the theorem.

Since $\X(\si)=G \times \De^k \simeq G$ it follows that $\map_G(\X(\si),Y) \simeq Y$ is path connected and simple.
Therefore $\pi_q \map_G(\X(\si),Y)$ vanishes if $q=0$, and for $q>0$ the basepoint $f|_{\si}$ is immaterial. 
For any $G$-set $\Omega$ and any $G$-module $M$ there is a natural isomorphism of abelian groups $\map_G(\Omega,M) \cong \Hom_{\ZZ G}(\ZZ[\Omega],M)$.
We therefore obtain natural isomorphisms for all $q \geq 1$
\begin{multline*}
\pi_q \big(\map_G(\X(\si),Y)_{f|_{\si}}\big) \cong 
[S^q,\map_G(\X(\si),Y) ] \cong  
[\X(\si),\map(S^q,Y)]_G  \\
\cong
\map_G(\pi_0 \X(\si),[S^q,Y]) \cong
\Hom_G(\ZZ[\pi_0\X(\si)], \pi_qY),
\end{multline*}
and for $q=0$ this isomorphism is trivial.
The proof is now complete since we have shown in \eqref{eq:limphg} that $\llim{D^{\op}}{}^p \Hom_G(\ZZ[\pi_0\X],\pi_qY) \cong H_G^p(X;\pi_qY)$.
\end{proof}

\begin{remark}\label{remark bfss natural}
The spectral sequence is natural in the following sense.
If $X' \subseteq X$ is a sub-polytope we obtain inclusion of posets $D_{X'} \subseteq D_X$ and $\X'=\X|_{D_{X'}}$.
The naturality of the Bousfield-Kan spectral sequence used in the proof, shows that there is an induced morphism on the spectral sequences.
\end{remark}

\begin{remarkx}
It is possible to deduce Theorem \ref{thm federer ss} from \cite[Theorem 1.1]{Schu73} by observing that $\map_G(X,Y)$ is homeomorphic to the space of sections of the fibration $X \times_G Y \to X/G$ and that $\pi_1(X/G)$ acts on the fibre $Y$ via the action of $G$ (note that $X \to X/G$ is a covering projection since the action of $G$ is free).
One then identifies the cohomology groups with local coefficients in Schultz's result with the Bredon cohomology groups.
\end{remarkx}

\section{The main result}
\label{sec main result}

Let $X$ be a compact polytope on which a finite group $G$ acts simplicially.
For every $n \geq 1$ let $X_n$ denote the $n$-fold join of $X$ with itself. 
The main result of this section is the following theorem.

\begin{thm}\label{thm pik ess bdd}
Let $X$ and $\{ X_n\}_n$ be as above.
Assume that for any $H \leq G$
\begin{itemize}
\item[(a)]
either $X^H$ is empty or it is homeomorphic to a sphere, and

\item[(b)]
the action of $N_GH$ on $H_*(X^H;\ZZ)$ is trivial. 
\end{itemize}
Then for any $k \geq 1$ the sequence $\{ \pi_k \aut_G(X^{*n})\}_n$ is essentially bounded.
\end{thm}

Theorem \ref{thm main them aut} follows from Theorem \ref{thm pik ess bdd} and Proposition \ref{prop lin spheres good}.

\begin{prop}\label{prop lin spheres good}
Let $X$ be a linear sphere $S(V)$ in a complex representation $V$ of $G$.
Then $X$ satisfies the hypothesis of Theorem \ref{thm pik ess bdd}.
\end{prop}

\begin{proof}
The action of $G$ on $S(V)$ factors through the action of $U(n)$ where $n=\dim V$.
Hence, $X$ is a sphere of dimension $2n-1$ on which $G$ acts smoothly.
By \cite{Illman78}, there is a triangulation of $X$ which renders the action of $G$ simplicial.
For any $H \leq G$ it is clear that $S(V)^H=S(V^H)$, so $X^H$ is either empty (if $V^H=0$) or it is a linear sphere, hence (a) holds. 
Also, $N_GH$ acts via isometries on $V^H$, namely the action factors through the unitary group $U(V^H)$ which is path connected, and therefore it acts on $S(V^H)$ via self-maps homotopic to the identity.
This proves (b).
\end{proof}

The next corollary to Theorem \ref{thm pik ess bdd} is a slight generalization of \cite[Lemma 2.6]{UY12}. 
For any $n  \geq 1$ there is a map $\de_n \colon \aut_G(X) \to \aut_G(X^{*n})$ given by $\vp \mapsto \vp^{*n}$.

\begin{cor}\label{cor null maps}
Fix $k \geq 1$.
Under the hypotheses of Theorem \ref{thm pik ess bdd}, for any $N$ there is $n \geq N$ such that $\pi_k \aut_G(X) \xto{(\de_n)_*} \pi_k\aut_G(X^{*n})$ is the trivial homomorphism.
\end{cor}

\begin{proof}
Recall that $X_n=X^{*n}$.
By Theorem \ref{thm pik ess bdd} there exists $M \geq 1$ such that $|\pi_k\aut_G(X_n)| \leq M$ for all $n \gg 0$.
We choose such an $n$ such that $n\geq N$ and $(M!)^2 | n$ and show that $\de_n$ satisfies the conclusion of the corollary.

The $n$-fold join $X^{*n}$ is the quotient space of $X^n\times \De^{n-1}$ where $\De^{n-1} \subseteq \RR^n$ is the standard $(n-1)$-simplex with the barycentric coordinates, by the equivalence relation $(x_1,\dots,x_n,t_1,\dots,t_n) \sim (x_1',\dots,x_n',t_1',\dots,t_n')$ if $t_i=t_i'$ for every $i$ and $x_i=x_i'$ unless $t_i=t_i'=0$.
The symmetric group $\Si_n$ acts $G$-equivariantly on $X_n=X^{*n}$ by permuting the factors of $X^n$.
The assignment $\vp \mapsto \si \circ \vp$ where $\si \in \Si_n$ and $\vp \in \aut_G(X_n)$ defines an action of $\Si_n$ on $\aut_G(X_n)$.
Since the component of $\id_{X_n}$ in $\aut_G(X_n)$ is an associative unital connected monoid, it is a simple space \cite[Corollary 3.6, p. 166]{Whitehead78} and hence $\pi_k \aut_G(X_n)$ (with the identity as basepoint) becomes a $\Si_n$-module.
We will write $\bullet$ for the monoidal operation in $\aut_G(X_n)$ and note that it is simply the composition of self-equivalences.
The group structure on $\pi_k\aut_G(X_n)$ coincides with the group structure on $[S^k,\aut_G(X_n)_{\id}]$ induced by the monoidal structure \cite[Theorem 5.21, p. 124]{Whitehead78}. 

Consider the map $\iota \colon \aut_G(X) \to \aut_G(X_n)$ defined by $\vp \mapsto \vp * 1_X * \dots * 1_X$.
If $\theta \colon S^k \to \aut_G(X)$ represents an element in $\pi_k \aut_G(X)$, then by inspection $\de_n \circ \theta$ is equal to  $(\tau_1 \circ \iota \circ \theta)\bullet \dots \bullet (\tau_n \circ \iota \circ \theta)$, where $\tau_1,\dots \tau_n$ are the elements of the cyclic group $C_n \leq \Si_n$.
Therefore
\[
(\de_n)_*([\theta])=[\underset{\tau\in C_n}{\bullet} (\tau \circ \iota \circ \theta)]=
\sum_{\tau \in C_n} \tau_*(\underbrace{\iota_*([\theta])}_{\omega}) =
\sum_{\tau \in C_n} \tau_*(\omega).
\]
Set $\Pi=\pi_k\aut_G(X_n)$.
By hypothesis $|\Pi| \leq M$ and since it is a $C_n$-module, we have a homomorphism $\rho \colon C_n \to \Aut(\Pi) \leq \Si_M$.
The order of the kernel of $\rho$ must be divisible by $(M!)^2/M!=M!$ which annihilates $\Pi$.
Therefore
\[
(\de_n)_*([\theta]) = \sum_{\tau \in C_n/\ker \rho} |\ker \rho| \cdot \tau_*(\omega) = 0.
\]
This completes the proof.
\end{proof}


\begin{prop}\label{prop abc 123}
Assume that the sequence $\{ X_n \}_n$ of $G$-spaces defined in the beginning of the section satisfies the conditions of Theorem \ref{thm pik ess bdd}.
Then it satisfies the conditions \ref{key I}--\ref{key III} of section \ref{sec key lemma}.
Moreover, for any $H \leq G$ and any $n \geq 1$, $(X_n)^H=(X^H)^{*n}$ and, if $H' \leq H$ then either $(X_n)^H=(X_n)^{H'}$ for all $n$ if $X^H=X^{H'}$ or $\dim (X_n)^{H'} - \dim (X_n)^H \geq n$ if $X^H \neq X^{H'}$.
%
\end{prop}

\begin{proof}
By construction of the join, $(X_n)^H=(X^{*n})^H=(X^H)^{*n}$.
By hypothesis (a) it is either empty if $X^H=\emptyset$ or it is homeomorphic to a sphere of dimension $nr+n-1$ if $X^H \cong S^r$.
By hypothesis (b) $N_GH$ acts on the sphere $X^H$ via maps of degree $1$, hence maps which are homotopic to the identity. 
Therefore $N_GH$ acts on $(X_n)^H$ via self-equivalences which are homotopic to the identity.
Conditions \ref{key II} and \ref{key I} follow.
%

Suppose that $H' \leq H$.
If $X^H \subseteq X^{H'}$ are spheres of different dimensions, then 
$\dim (X_n)^{H'}-\dim (X_n)^H=n(\dim X^{H'}-\dim X^H) \geq n$.
%
If they have the same dimension then they must be equal by the invariance of domain.
\end{proof}

\begin{proof}[Proof of Theorem \ref{thm pik ess bdd}]
The groups $\pi_1\aut_G(X_n)$ are abelian since $\aut_G(X_n)$ are associative unital monoids.
For any subgroup $H \leq G$ let $X_n^{(H)}$ denote the subspace of $X_n$ which consists of the points $x \in X_n$ whose isotropy group $G_x$ is conjugate to $H$.
Arrange the conjugacy classes of the subgroups of $G$ in decreasing order $(H_1), (H_2), \dots ,(H_m)=(e)$, namely if $i>j$ then $|H_i| \leq |H_j|$.
We will now define filtrations $F_0(X_n) \subseteq \dots \subseteq F_m(X_n)$ of the spaces $X_n$.
For every $s \geq 0$ set
\[
F_s(X_n) = \bigcup_{i=1}^s X_n^{(H_i)}.
\]
We will prove by induction on $s$ that for any $k \geq 1$ the sequence of groups 
\begin{equation}\label{eq induct}
\{ \pi_k \map_G(F_s(X_n),X_n)_{\incl} \}_{n =1}^\infty  \qquad \text{is essentially bounded.}
\end{equation}
This will complete the proof since $F_m(X_n)=X_n$.
If $s=0$ then $F_s(X_n)$ are empty and \eqref{eq induct} is trivial.
We now fix some $1 \leq s \leq m$ and prove the induction step. 

Set $K:=H_s$ and $W=N_G(K)/K$.
Set $Y_n=X_n^K$ and $A_n= X_n^{>K}$ both viewed as $W$-spaces.
By possibly passing to the barycentric subdivision of $X_n$ we may assume that $Y_n$ and $A_n$ are sub-polytopes of $X_n$ on which $W$ acts simplicially.
Remark \ref{rem x>K 1} implies that $(Y_n)^{>e} \subseteq A_n$.
Also, $A_n=F_{s-1}(X_n)^{K}$ and $Y_n=F_s(X_n)^K$ because for any $x \in X_n$ we must have $x \in Y_n$ (resp. $x \in A_n$) if and only if $(G_x)=(H_i)$ for some $i \leq s$ (resp. $i<s$). 

Since $G \times_{N_GK} (G/K)^K=G/K$ and since $F_s(X_n)\backslash F_{s-1}(X_n)$ consists only of orbits isomorphic to $G/K$, we obtain the following pushout square in which the vertical maps are inclusion of $G$-CW complexes, in fact inclusions of $G$-simplicial complexes, hence $G$-cofibrations:
\[
\xymatrix{
G \times_{N_GK} F_{s-1}(X_n)^K \ar[rr]^(0.6){\ev \colon (g,x) \mapsto gx} \ar@{^(->}[d] & &
F_{s-1}(X_n)  \ar@{^(->}[d] \\
G \times_{N_GK} F_s(X_n)^K \ar[rr]^(0.6){\ev\colon (g,x) \mapsto gx}  & &
F_s(X_n).
}
\]
Suppose first that $X^K=X^H$ for some $H>K$.
By Proposition \ref{prop abc 123}, $(X_n)^K=(X_n)^H$ for all $n$ and therefore $A_n=Y_n$ for all $n$.
It follows from the pushout diagram that $F_s(X_n)=F_{s-1}(X_n)$ for all $n$ and the induction step for \eqref{eq induct} follows from the induction hypothesis on $s-1$.

Hence we will assume from now on that $X^K \neq X^H$ for all $H>K$, so by Proposition \ref{prop abc 123} $\dim Y_n - \dim A_n  \geq n$ and in particular $A_n \subsetneq Y_n$ for all $n$.
By applying $\map_G(-,X_n)$ to the pushout diagram above, we obtain the following pullback diagram in which the horizontal arrows are fibrations
\begin{equation}\label{mainth:e1}
\xymatrix{
\map_G(F_s(X_n),X_n) \ar@{->>}[r] \ar[d] &
\map_G(F_{s-1}(X_n),X_n) \ar[d] 
\\
\map_{N_GK}(Y_n,X_n) \ar@{->>}[r] &
\map_{N_GK}(A_n,X_n).
}
\end{equation}
This is therefore a homotopy pullback square so the fibres of the rows are weakly homotopy equivalent (in fact, they are homeomorphic).
Let $F'_n$ denote the fibre of the bottom row over the inclusion $i_{A_n}^{X_n}$ and with the inclusion $i_{Y_n}^{X_n}$ as a basepoint.
By induction hypothesis $\{\pi_k\map_G(F_{s-1}(X_n),X_n)_{\incl}\}_n$ is essentially bounded for any $k \geq 1$, and from the long exact sequence in homotopy of the fibration in the first row, it suffices to prove that the sequence $\{\pi_k F'_n\}_n$ is essentially bounded for any $k \geq 1$.

By possibly passing to the barycentric subdivision of each $Y_n$, we can choose, by Lemma \ref{lem abdy}, for any $n$ a $W$-invariant neighbourhood $U_n$ of $A_n$ in $Y_n$ such that $\overline{U_n}$ and $B_n:=Y_n \backslash U_n$ are sub-polytopes of $Y_n$, and the inclusions $A_n \subseteq U_n \subseteq \overline{U_n}$ and $Y_n \backslash \overline{U_n} \subseteq B_n \subseteq Y_n \backslash A_n$ are $W$-equivalences.
Set $D_n=\overline{U_n} \cap B_n$.

Since $A_n$ and $Y_n=X_n^K$ are fixed by $K$, the bottom row of \eqref{mainth:e1} is the same as the fibration  $\map_W(Y_n,Y_n) \to \map_W(A_n,Y_n)$.
Also, $A_n \subseteq \overline{U_n}$ is a $W$-homotopy equivalence and therefore $F'_n$ is homotopy equivalent to the fibre $F_n$ of the following fibration over $\incl \in \map_W(\overline{U_n},Y_n)$ 
\[
F_n \to \map_W(Y_n,Y_n) \to \map_W(\overline{U_n},Y_n).
\] 
We will complete the proof by showing that $\{\pi_k F_n\}_n$ is essentially bounded for any $k \geq 1$.

By construction, 
$Y_n$ is the pushout of the sub-complexes $B_n$ and $\overline{U_n}$ along $D_n$.
By applying $\map_W(-,Y_n)$ we obtain a pullback diagram in which all the arrows are fibrations.
\[
\xymatrix{
\map_W(Y_n,Y_n) \ar@{->>}[r] \ar[d] &
\map_W(\overline{U_n},Y_n) \ar[d] 
\\
\map_W(B_n,Y_n) \ar@{->>}[r] &
\map_W(D_n,Y_n)
}
\]
It is therefore a homotopy pullback square, hence $F_n$ is weakly equivalent to the fibre of the bottom row over the inclusion $i_{D_n}^{Y_n}$ (in fact the fibres are homeomorphic).
We obtain exact sequences of groups (one for each $n$)
\begin{multline*}
\pi_{k+1}\map_W(B_n,Y_n)_{\incl} \xto{\la_{k+1}} \pi_{k+1}\map_W(D_n,Y_n)_{\incl} \to \pi_k(F_n) \to 
\\
\to \pi_k\map_W(B_n,Y_n)_{\incl} \xto{\la_k} \pi_k\map_W(D_n,Y_n)_{\incl}.
\end{multline*}
We see that to complete the proof it is enough to show that for any $k \geq 1$ the sequences $\{\ker(\la_k)\}_n$ and $\{\coker(\la_{k+1})\}_n$ are essentially bounded.
This will be the goal of the remainder of the proof.

Since $A_n \supseteq Y_n^{>e}$, it follows that $W$ acts freely on $B_n$ and $D_n$.
By Proposition \ref{prop abc 123} and hypothesis (a) of the theorem, $Y_n=X_n^K$ are spheres of dimension $N=n\dim(X^K)+n-1$. 
In particular, $Y_n$ are simple spaces.
By appealing to Theorem \ref{thm federer ss}  we obtain Barrat--Federer spectral sequences of the form
\begin{eqnarray*}
E^2_{-p,q}(B_n) &=& H_W^{p}(B_n;\pi_q Y) \Rightarrow \pi_{q-p} \map_W(B_n,Y_n)_{\incl} \\
E^2_{-p,q}(D_n) &=& H_W^{p}(D_n;\pi_q Y_n) \Rightarrow \pi_{q-p} \map_W(D_n,Y_n)_{\incl} 
\end{eqnarray*}
By Remark \ref{remark bfss natural} the inclusions $D_n \subseteq B_n$ induce morphisms of spectral sequences 
\[
\theta^r \colon E^r_{*,*}(B_n) \to E^r_{*,*}(D_n).
\]
For dimensional reasons $E^2_{-p,*}(B_n)=0$ and $E^2_{-p,*}(D_n)=0$ if $p>N$.
Also $E^2_{*,q}(B_n)=0$ and $E^2_{*,q}(D_n)=0$ if $q<N$ because $Y_n$ is $(N-1)$-connected.

By Proposition \ref{prop abc 123}, $W$ acts trivially on $\pi_*Y_n$ since condition \ref{key I} of Section \ref{sec key lemma} holds.
Therefore the coefficient systems in the $E^2$-terms of these spectral sequences are trivial.
In addition, once $k$ is fixed then for every $0 \leq j \leq k$  the groups $\pi_{N+j} Y_n=\pi_{N+j} S^N$ enter the stable range provided $n \gg 0$ so we may assume that $\pi_{N+j} Y_n \cong \pi^S_{j}$, which are finite if $j\geq 1$.
By Proposition \ref{prop abc 123} we may apply Lemma  \ref{lemma key}\ref{key C} which implies that for any $N+1 \leq q \leq N+k$ and any $N-k \leq p \leq N$, the sequences $\{E^2_{-p,q}(B_n)\}_n$ and $\{E^2_{-p,q}(D_n)\}_n$ are essentially bounded.
It follows that $\{E^\infty_{-p,q}(B_n)\}_n$ and $\{E^\infty_{-p,q}(D_n)\}_n$ are essentially bounded for $p,q$ in this range.
The spectral sequences $E^r_{*,*}(B_n)$ and $E^r_{p,q}(D_n)$ are depicted in Figure \ref{fig spectral sequence 1}.
\begin{figure}
\includegraphics[scale=0.6]{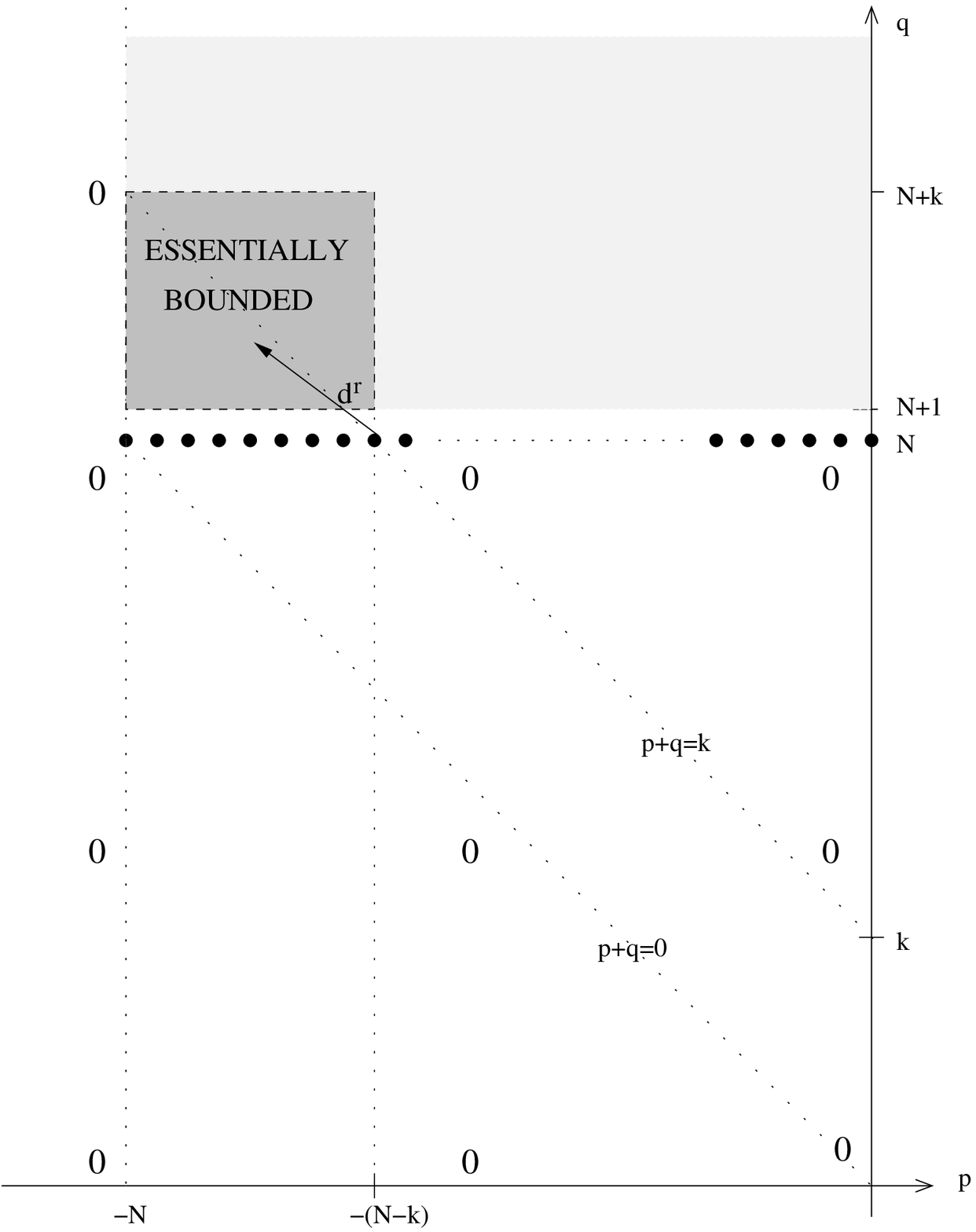}
\caption{The spectral sequences $E^r_{p,q}(B_n)$ and $E^r_{p,q}(D_n)$. 
Differentials have degree $(-r,r-1)$.}
\label{fig spectral sequence 1}
\end{figure}
%
We deduce that $\pi_k \map_W(B_n,Y_n)_{\incl}$ has a filtration of length $k+1$ with  filtration quotients 
\[
E^\infty_{-(N-k),N}(B_n),\ E^\infty_{-(N-k+1),N+1}(B_n),\ \dots \ ,\ E^\infty_{-N,N+k}(B_n)
\] 
all except the first are essentially bounded (as sequences of groups indexed by $n$).
A similar statement holds for $\pi_k \map_W(D_n,Y)_{\incl}$.
We obtain the following diagram of short exact sequences (indexed by $n$)
\begin{equation}\label{mainth:e3}
\xymatrix{
0 \ar[r] &
\text{ess. bounded} \ar[r] \ar[d] &
\pi_k\map_W(B_n,Y_n) \ar[d]^{\la_k} \ar[r] &
E^\infty_{-(N-k),N}(B_n) \ar[d]^{\th^\infty_{-(N-k),N}} \ar[r] &
0
\\
0 \ar[r] &
\text{ess. bounded} \ar[r]  &
\pi_k\map_W(D_n,Y_n) \ar[r] &
E^\infty_{-(N-k),N}(D_n) \ar[r] &
0
}
\end{equation}
So we only need to show that $\{\ker(\th_{-(N-k),N}^\infty)\}_n$ is essentially bounded if $k \geq 1$ and that $\{\coker(\th_{-(N-k),N}^\infty)\}_n$ is essentially bounded if $k \geq 2$.
Since the spectral sequences vanish for $q<N$ and all the differentials $d^r_{-(N-k),N}$ have their target in essentially bounded groups, we obtain the following diagram of short exact sequences.
\[
\xymatrix{
0 \ar[r] &
E^\infty_{-(N-k),N}(B_n) \ar[d]^{\th^\infty_{-(N-k),N}} \ar[r] &
E^2_{-(N-k),N}(B_n) \ar[d]^{\th^2_{-(N-k),N}} \ar[r] &
\text{ess. bounded} \ar[d] \ar[r] &
0
\\
0 \ar[r] &
E^\infty_{-(N-k),N}(D_n) \ar[r] &
E^2_{-(N-k),N}(D_n)  \ar[r] &
\text{ess. bounded} \ar[r] &
0
}
\]
Lemma \ref{lemma key}\ref{key D} implies that if $k \geq 1$ then $\{\ker(\th^2_{-(N-k),N})\}_n$ is essentially bounded and therefore so is $\{\ker(\th_{-(N-k),N}^\infty)\}_n$.
By the same lemma, if $k \geq 2$ then the sequence of groups  $\{\coker(\th^2_{-(N-k-1),N})\}_n$ is essentially bounded and therefore so is $\{\coker(\th_{-(N-k),N}^\infty)\}_n$.
This completes the proof.
\end{proof}


\end{document}